\documentclass[12pt,a4paper]{amsart}
\usepackage{t1enc}
\usepackage[latin1]{inputenc}
\usepackage{amssymb}
\usepackage{amsthm}
\usepackage{amsmath}
\usepackage{latexsym}
\usepackage{array}
\usepackage{mathrsfs}
\usepackage[all]{xy}
\usepackage{graphics}
\usepackage{xcolor}
\usepackage{stmaryrd}
\title{{AROUND RATIONALITY OF INTEGRAL CYCLES}}
\date{11 March 2012}
\author{{Raphaël Fino}}
\address
{UPMC Sorbonne Universit\'es\\
Institut de Math\'ematiques de Jussieu\\
Paris\\
FRANCE}
\address
{{\it Web page:}
{\tt www.math.jussieu.fr/\~{ }fino}}
\email {fino {\it at} math.jussieu.fr}
\numberwithin{equation}{section}
\theoremstyle{definition}
\newtheorem{defi}[equation]{Definition}
\newtheorem{rem}[equation]{Remark}
\newtheorem{lemme}[equation]{Lemma}

\newtheorem{prop}[equation]{Proposition}
\newtheorem{thm}[equation]{Theorem}

\begin{document}

\begin{abstract}
In this article we prove a result comparing rationality of integral algebraic cycles over the function field of a quadric and over the base field. This is an integral version of the result known for $\mathbb{Z}/2\mathbb{Z}$-coefficients.
Those results have already been proved by Alexander Vishik in the case of characteristic $0$, which allowed him to work with algebraic cobordism theory. Our proofs use the modulo $2$ Steenrod operations in the Chow theory and work in any characteristic $\neq 2$. 
\end{abstract}

\maketitle

\medskip

\medskip

In many situations it can be important to know, if an element of the Chow group of some variety considered over an algebraic closure of its base field, is actually defined over the base field itself. 
In a previous paper (\cite{ARC1}), we already dealt with Chow groups modulo~$2$. In particular, we showed that for a Chow group modulo~$2$ element
of codimension $m$, it is sufficient to check that it is defined over the function field of a sufficiently large quadric $Q$
(of dimension $>2m$), to get that it is defined over the base field (up to an element of exponent $2$, see \cite[Theorem 1.1]{ARC1}).
Furthermore, we did this without the help of algebraic cobordism theory (it was the case in \cite{GPQ}) but with the only help of the Chow theory itself (including the \textit{Steenrod operations on Chow groups modulo $2$}). 
This allowed one to get a valid result in any characteristic different from $2$
(however, the use of \textit{symmetric operations} in algebraic cobordism theory by Alexander Vishik in \cite{GPQ} permitted him to
obtain a statement without an exponent $2$ element).

In this note, we prove an integral version of \cite[Theorem 1.1]{ARC1} (generalizing in a way \cite[Theorem 3.1]{RIC} to any characteristic different from $2$).
We work again with the only help of Steenrod operations on Chow groups modulo $2$, which is remarkable.
This integral version (Theorem 3.1), whose the statement is very close to the $\mathbb{Z}/2\mathbb{Z}$-coefficients case,
requires an additional condition on $Q$ (aside from its size) saying that $Q$ has a projective line
defined over the generic point of $Q$.

I would like to gratefully thank Nikita Karpenko for sharing his great knowledge and his valuable advice. 
This work could not have been done without his help.

\tableofcontents

\section{Preliminaries: decomposition of Chow groups}

In this paper, the word \textit{scheme} means a separated scheme
of finite type over a field and a \textit{variety} is an integral scheme. 
Let $F$ be a field of characteristic $\neq 2$ and $Y$ be an $F$-variety. 
We write $CH(Y)$ for the integral Chow group of $Y$ (see \cite[Chapter X]{EKM}) and we write $Ch(Y)$ for $CH(Y)$ modulo $2$.

\medskip

The main purpose of this section is to introduce the notion of \textit{coordinates} for a cycle $x \in CH(Q\times Y)$, 
where $Q$ is a smooth projective quadric over $F$ and $Y$ is a smooth $F$-variety. 
This notion will be useful during the proofs of Theorem 3.1 and Theorem 4.1.

\medskip

Let $Q$ be a smooth projective quadric over $F$ of dimension $n$ given by a quadratic form $\varphi$, and
let us set $i_0(Q):=i_0(\varphi)$, where $i_0(\varphi)$ is the Witt index of $\varphi$. 

\medskip

For $i=0,...,n$, let us denote as $h^i \in CH^i(Q)$ the $i$th power of the hyperplane section class
(note that for any $i$, the cycle $h^i$ is defined over the base field).
For $i<i_0(Q)$, let us denote as $l_i \in CH_i(Q)$ the class of an $i$-dimensional totally isotropic subspace of $\mathbb{P}(V)$,
where $V$ is the underlying vector space of $\varphi$.
For $i\leq [n/2]$, we still write $l_i \in CH_i(Q_{\overline{F}})$ for the class of an $i$-dimensional totally isotropic subspace of $\mathbb{P}(V_{\overline{F}})$, where $\overline{F}$ is an algebraic closure of $F$
(if $i<i_0(Q)$, the cycle $l_i \in CH_i(Q_{\overline{F}})$ is the image of $l_i \in CH_i(Q)$
under the change of field homomorphism $CH(Q)\rightarrow CH(Q_{\overline{F}})$).
Let us notice that for $i<[n/2]$, the cycle $l_i$ (in $CH_i(Q_{\overline{F}})$ or in $CH_i(Q)$ if $i<i_0(Q)$)
is canonical by \cite[Proposition 68.2]{EKM} (in case of even $n$, there are two classes
of $n/2$-dimensional totally isotropic subspaces and we fix one of the two).

\medskip

Let $x$ be an element of $CH^r(Q\times Y)$.
We write $pr$ for the projection $Q\times Y \rightarrow Y$.
For every $i=0,...,i_0(Q)-1$, we have the following homomorphisms
\[\begin{array}{rcl}
CH^r(Q\times Y)  & \longrightarrow & CH^{r-i}(Y) \\
x & \longmapsto & pr_{\ast}(l_i\cdot x)=:x^i
   \end{array},\]
and 
\[\begin{array}{rcl}
CH^r(Q\times Y)  & \longrightarrow & CH^{r-n+i}(Y) \\
x & \longmapsto & pr_{\ast}(h^i\cdot x)=:x_i
   \end{array}.\]
   
\begin{defi}
The cycle $x^i\in CH^{r-i}(Y)$ is called the \textit{coordinate of $x$ on $h^i$} while  
$x_i\in CH^{r-n+i}(Y)$ is called the \textit{coordinate of $x$ on $l_i$}.
\end{defi}

Note that if $r<[n/2]$, for any $i=0,...,i_0(Q)-1$, one has $x_i=0$ by dimensional reasons.

\begin{rem}
For any nonnegative integer $k<i_0(Q)$, let us set $x(k):=x-\sum_{i=0}^k h^i\times x^i -\sum_{i=0}^k l_i\times x_i$.
Note that for any $i=0,...,k$, the coordinate of $x(k)$ on $h^i$ (as well as its coordinate on $l_i$) is $0$.
The writing 
\[x=x(k)+\sum_{i=0}^k h^i\times x^i+\sum_{i=0}^k l_i\times x_i\]
is called a \textit{decomposition} of $x$.

Assume now that $r<i_0(Q)$ and $r\leq k$. Then, by \cite[Theorem 66.2]{EKM}, one can write 
\[x(k)=\sum_{i=0}^r h^i\times w^i\]
with some $w^i\in CH^{r-i}(Y)$.
Since, for any $i=0,...,r$, the cycle $w^i$ coincides with the coordinate $x(k)^i$ of $x(k)$ on $h^i$,
we get that $x(k)=0$.
\end{rem}

Recall that one says that the quadric $Q$ is \textit{completly split} if $i_0(Q)=[n/2]+1$. 

\begin{rem}
Assume that $Y=Q$, $r=n$, and that $k<[n/2]$ (what is the case if the quadric $Q$ is not completly split).
Let $x$ be an element of $CH^{n}(Q\times Q)$.
Since, for $i=0,...,k$, the group $CH^{n-i}(Q)$ is free with basis $\{l_i\}$ (because $i<[n/2]$, see \cite[\S 68]{EKM}),
one can uniquely write 
\[x= x(k)+\sum_{i=0}^{k}b_i \cdot h^i\times l_i+\sum_{i=0}^{k}l_i\times x_i,\]
with some $b_i\in \mathbb{Z}$.
\end{rem}

Note that everything in Section~1 holds for Chow groups modulo $2$ in place of the integral Chow groups.

\section{Preliminaries: Steenrod operations and correspondences}
In this section we continue to use notation introduced in the beginning of Section~1.

\medskip

The Steenrod operations are the main tool of this note. We refer to \cite[Chapter XI]{EKM} for an introduction to the subject.
We just recall here that for a smooth scheme $X$ over a field $F$ (of characteristic $\ne2$), there is a certain homomorphism
$S_X:CH(X)\rightarrow CH(X)$ called the \textit{total Steenrod operation on $X$ of cohomological type}. 
Although $S_X$ is constructed only for quasi-projective $X$ in \cite{EKM},
Patrick Brosnan has extended the operation $S_X$ to any scheme $X$ in \cite[\S 10]{Bros}. In particular, the following proposition
holds for any smooth scheme (not necessarily quasi-projective).
This allows us to get rid of the assumption of quasi-projectivity for the main theorem of this note (Theorem 3.1)
(Alexander Vishik needed that assumption in \cite{RIC} because the algebraic cobordism theory is defined on the category
of smooth quasi-projective schemes over $F$, see \cite{AC}).

\medskip

In the following proposition, whose the statement and the proof are very close to \cite[Lemma 3.1]{RPSO},
 we focus on how the Steenrod operations interact with the composition of correspondences
(correspondences are defined in \cite[\S 62]{EKM}).
This will be useful during the proof of Theorem 3.1.

\medskip

Let $X_1$, $X_2$, $X_3$ be smooth schemes over $F$ (of characteristic $\neq 2$), and assume that $X_2$ is complete
(so the push-forward associated with the projection $X_1 \times X_2 \times X_3 \longrightarrow X_1 \times X_3$ is well defined).

\begin{prop}
\textit{For any correspondence $\alpha \in CH(X_1 \times X_2)$ and for any correspondence $\beta \in CH(X_2 \times X_3)$, one has}

\medskip
 
\textit{1) $S_{X_1\times X_3}(\beta \circ \alpha)=(S_{X_2\times X_3}(\beta)\cdot c(-T_{X_2}))\circ S_{X_1\times X_2}(\alpha);$}

\medskip

\textit{2) $S_{X_1\times X_3}(\beta \circ \alpha)=S_{X_2\times X_3}(\beta)\circ (S_{X_1\times X_2}(\alpha)\cdot c(-T_{X_2})),$}

\medskip

\textit{where $T_{X_2}$ is the tangent bundle of $X_2$ and $c$ is the total Chern class.}
\end{prop}

\begin{proof}
For any $i,j \in \{1,2,3\}$ such that $i<j$, let us write $p_{ij}$ for the projection
\[X_1 \times X_2 \times X_3 \longrightarrow X_i \times X_j.\]

According to the composition rules of correspondences described in \cite[\S 62]{EKM}, we have
\[\beta \circ \alpha = {p_{13}}_{\ast}({p_{12}}^{\ast}(\alpha)\cdot {p_{23}}^{\ast}(\beta)).\]

Therefore, by \cite[Proposition 61.10]{EKM} applied to $p_{13}$, we get
\[S_{X_1\times X_3}(\beta \circ \alpha)=
{p_{13}}_{\ast}(S_{X_1 \times X_2 \times X_3}({p_{12}}^{\ast}(\alpha)\cdot {p_{23}}^{\ast}(\beta))\cdot {p_{12}}^{\ast}({pr_2}^{\ast}(c(-T_{X_2})))),\]
and since $S$ conmmutes with the products and the pull-backs, we get
\[S_{X_1\times X_3}(\beta \circ \alpha)=
{p_{13}}_{\ast}({p_{12}}^{\ast}(S_{X_1 \times X_2}(\alpha))\cdot {p_{23}}^{\ast}(S_{ X_2\times X_3}(\beta))\cdot ([X_1]\times c(-T_{X_2})\times [X_3] )),\]
this gives, on the one hand 
\[S_{X_1\times X_3}(\beta \circ \alpha)=
{p_{13}}_{\ast}({p_{12}}^{\ast}(S_{X_1 \times X_2}(\alpha))\cdot {p_{23}}^{\ast}(S_{ X_2\times X_3}(\beta)\cdot c(-T_{X_2}))),\]
thus 1) is proved, and on the other hand, this gives
\[S_{X_1\times X_3}(\beta \circ \alpha)=
{p_{13}}_{\ast}({p_{12}}^{\ast}(S_{X_1 \times X_2}(\alpha)\cdot c(-T_{X_2}))\cdot {p_{23}}^{\ast}(S_{ X_2\times X_3}(\beta))),\]
thus 2) is proved.
\end{proof}

\section{Main theorem}
In this section we continue to use notation introduced in the beginnings of Sections~1 and 2.

\medskip

Let $F$ be a field of characteristic $\neq 2$ and let $Y$ be a smooth $F$-variety.
We write $\overline{Y}:=Y_{\overline{F}}$ where $\overline{F}$ is an algebraic closure of $F$.
Let $X$ be a geometrically integral variety over $F$. An element $\overline{y}$ of $CH(\overline{Y})$ (or of $Ch(\overline{Y})$)
is \textit{$F(X)$-rational} if its image $\overline{y}_{\overline{F}(X)}$ under $CH(\overline{Y})\rightarrow CH(Y_{\overline{F}(X)})$
(resp. $Ch(\overline{Y})\rightarrow Ch(Y_{\overline{F}(X)})$) is in the image of $CH(Y_{F(X)})\rightarrow CH(Y_{\overline{F}(X)})$ 
(resp. $Ch(Y_{F(X)})\rightarrow Ch(Y_{\overline{F}(X)})$). An element $\overline{y}$ of $CH(\overline{Y})$ (or of $Ch(\overline{Y})$)
is called \textit{rational} if it belongs to the subgroup $\overline{CH}(Y):=$Im$\left( CH(Y)\rightarrow CH(\overline{Y}) \right)$ 
(resp. $\overline{Ch}(Y)$).
Note that since $\overline{F}$ is algebraically closed, the homomorphism $CH(\overline{Y})\rightarrow CH(Y_{\overline{F}(X)})$
(as well as the homomorphism $Ch(\overline{Y})\rightarrow Ch(Y_{\overline{F}(X)})$) is injective by the specialization arguments.

\medskip

Let $Q$ be a smooth projective quadric over $F$ of positive dimension $n$
(in that case, $Q$ is geometrically integral) given by a quadratic form $\varphi$.
Since for isotropic $Q$, any $F(Q)$-rational element (in any codimension) is rational, we make the assumption
that the quadric $Q$ is anisotropic. In particular, $Q$ is not completly split and 
one can consider the first Witt index $i_1(\varphi)$ of $\varphi$, which we simply denote as $i_1$.

\medskip

In a way, the following result is a generalization of \cite[Theorem 3.1]{RIC}. Indeed, the use of the Steenrod operations on the modulo $2$ Chow groups allows one to obtain a valid result in any characteristic different from $2$. Nevertheless, an exponent $2$ element appears in our conclusion while it is not the case in \cite[Theorem 3.1]{RIC}.
The proof is inspired by the proof of \cite[Theorem 3.1]{RIC}.

\begin{thm}
\textit{Assume that $m<[n/2]$ and $i_1>1$. Then any $F(Q)$-rational element of $CH^m(\overline{Y})$ is the sum of a rational element and an 
exponent $2$ element.}
\end{thm}

\begin{proof}
The statement being trivial for negative $m$, we may assume that $m\geq 0$ in the proof. Let $\overline{y}$ be an $F(Q)$-rational element of $CH^m(\overline{Y})$. 
Since the quadric $Q$ is isotropic over $\overline{F}$, the homomorphism $CH(\overline{Y})\rightarrow CH(Y_{\overline{F}(Q)})$
is surjective and is consequently an isomorphism.
The element $\overline{y}\in CH^m(\overline{Y})$ being $F(Q)$-rational, there exists $y \in CH^m(Y_{F(Q)})$ mapped to $\overline{y}$ under the homomorphism  
\[CH^m(Y_{F(Q)})\rightarrow CH^m(Y_{\overline{F}(Q)})\xrightarrow{\sim} CH^m(\overline{Y}).\]

Let us fix an element $x \in CH^m(Q\times Y)$ mapped to $y$ under the surjection (see \cite[Corollary 57.11]{EKM})
\[CH^m(Q\times Y)\twoheadrightarrow CH^m(Y_{F(Q)}).\] 
Since over $\overline{F}$ the quadric $Q$ becomes completly split and $m<[n/2]$, by Remark 1.2 (applied with $r=k=m$), the image $\overline{x}\in CH^m(\overline{Q}\times \overline{Y})$ 
of $x$ decomposes as
\begin{equation} \overline{x}=\sum_{i=0}^m h^i \times x^i \end{equation}
where $x^i \in CH^{m-i}(\overline{Y})$ is the coordinate of $\overline{x}$ on $h^i$ (see Definition 1.1), and where $x^0=\overline{y}$
by \cite[Lemma 3.2]{GPQ}.

\medskip

Let $\pi \in \overline{Ch}_{n+i_1-1}(Q^2)$ be the $1$-primordial cycle (see \cite[Definition 73.16]{EKM} and paragraph right after 
\cite[Theorem 73.26]{EKM}).
Since $i_1>1$, by \cite[Proposition 83.2]{EKM}, we get that the cycle $(h^0\times h^{i_1-1})\cdot \pi \in \overline{Ch}_{n}(Q^2)$ decomposes as
\begin{equation}(h^0\times h^{i_1-1})\cdot \pi=\sum_{p=0}^r \varepsilon_p (h^{2p}\times l_{2p})
+\sum_{p=0}^r \varepsilon_p (l_{2p+i_1-1}\times h^{2p+i_1-1}),\end{equation}
where $\varepsilon_p \in \{0,1\}$, $\varepsilon_0=1$, and $r=[\frac{d-i_1+1}{2}]$ with $d=[\frac{n}{2}]$. 
Thus, one can choose a rational integral representative $\overline{\gamma} \in \overline{CH}_n(Q^2)$ of $(h^0\times h^{i_1-1})\cdot \pi$
such that $\overline{\gamma}$ decomposes as 
\begin{equation}\overline{\gamma}=\sum_{i=0}^{[\frac{n}{2}]}\alpha_i(h^i\times l_i)+\sum_{i=0}^{[\frac{n}{2}]}\beta_i(l_i\times h^i)
+\delta(l_{[\frac{n}{2}]}\times l_{[\frac{n}{2}]}),
\end{equation}
with some integers $\alpha_i$, $\beta_i$ and $\delta$, where $\alpha_i$ is even for all odd $i$ and $\alpha_0$ is odd.

The element $\overline{\gamma}$ being rational, there exists $\gamma \in CH_n(Q^2)$ mapped to $\overline{\gamma}$ under the restriction homomorphism $CH_n(Q^2)\rightarrow CH_n(\overline{Q}^2)$. The cycles $\gamma$ and $\overline{\gamma}$ are considered here as correspondences of degree $0$.

\begin{lemme}
\textit{For any $i=0,...,m$, one can choose a rational integral representative $s^i\in CH^{m+i}(\overline{Q}\times \overline{Y})$ of $S^i((\overline{x}\;mod\;2)\circ (\overline{\gamma}\;mod\;2))$ such that}

\medskip

\textit{1) for any $0\leq j \leq m$, $2s^{i,j}$ is rational , where $s^{i,j}\in CH^{m+i-j}(\overline{Y})$ is the coordinate of $s^i$ on $h^j$;}

\medskip

\textit{2) for any odd $0\leq j \leq m$, $s^{i,j}$ is rational.}
\end{lemme}

\begin{proof}
First of all, since $m<[n/2]$, for any $j=0,...,m$, one has $h^{n-j}=2l_j$ (see \cite[\S 68]{EKM}).
Therefore, for any rational cycle $s\in CH(\overline{Q}\times \overline{Y})$, the element 
$2pr_{\ast}(l_j\cdot s)$ (where $pr$ is the projection $Q\times Y \rightarrow Y$) is rational and 1) is proved.

\medskip

Assume now that $j$ is odd. 
By Proposition 2.1 1), for any $i=0,...,m$, one has
\begin{equation}
S^i((\overline{x}\;mod\;2)\circ (\overline{\gamma}\;mod\;2))=\sum_{k=0}^m \sum_{t=0}^m (S^t(\overline{x}\;mod\;2)\cdot c_{i-k-t}(-T_Q))
\circ S^k(\overline{\gamma}\;mod\;2).
\end{equation} 

For every $k=0,...,m$, let $\tilde{a}^k \in CH^{n+k}(\overline{Q}\times \overline{Q})$ be a rational integral representative of 
$S^k(\overline{\gamma}\;mod\;2) \in Ch^{n+k}(\overline{Q}\times \overline{Q})$. 
We write $\tilde{a}^{k,j}\in CH^{n+k-j}(\overline{Q})$ for the coordinate of $\tilde{a}^k$ on $h^j$.
For every $k=0,...,m$ and every $t=0,...,m$, we choose  a rational integral
representative $d_{k,t} \in CH^{m+i-k}(\overline{Q}\times \overline{Y})$ of $S^t(\overline{x}\;mod\;2)\cdot c_{i-k-t}(-T_Q)\in Ch^{m+i-k}(\overline{Q}\times \overline{Y})$. Thus, by the equation (3.6), the cycle 
\[s^i:=\sum_{k=0}^m \sum_{t=0}^m d_{k,t}\circ \tilde{a}^k \in CH^{m+i}(\overline{Q}\times \overline{Y})\]
is a rational integral representative of $S^i((\overline{x}\;mod\;2)\circ (\overline{\gamma}\;mod\;2))$.

\medskip

Moreover, for any $0\leq k \leq m$, one has by (3.3)
\[S^k(\overline{\gamma}\;mod\;2)=\sum_{p=0}^r \varepsilon_p S^k(h^{2p}\times l_{2p})
+\sum_{p=0}^r \varepsilon_p S^k(l_{i_1-1+2p}\times h^{i_1-1+2p}).\]

Therefore, for any $0\leq k \leq m$, denoting as $a^{k,j}\in Ch^{n+k-j}(\overline{Q})$ the coordinate of $S^k(\overline{\gamma}\;mod\;2)$
on $h^j$, we have
\[a^{k,j}=\sum_{(p,t)\in \mathcal{E}_{k,j}} \varepsilon_p {2p \choose t} S^{k-t}(l_{2p})\]
where $\mathcal{E}_{k,j}=\{(p,t)\in \llbracket 0,r  \rrbracket \times \llbracket  0,k  \rrbracket \;| 2p+t=j\}$.

Furthermore, since $j$ is odd, for any $(p,t)\in \mathcal{E}_{k,j}$, the binomial coefficient ${2p \choose t}$ is even.
Therefore, for any $0\leq k \leq m$, we have $a^{k,j}=0$ and, consequently, the cycle $\tilde{a}^{k,j}\in CH^{n+k-j}(\overline{Q})$ is divisible by $2$.  
Since $j-k<[n/2]$, the group $CH^{n+k-j}(\overline{Q})$ is generated by $l_{j-k}$ and $2l_{j-k}=h^{n+k-j}$ (see \cite[\S 68]{EKM}).
Hence, for any $0\leq k \leq m$, the cycle $\tilde{a}^{k,j}$ is rational.

\medskip

According to the composition rules of correspondences described in \cite[\S 62]{EKM}, we have the identity
\[h^j\times s^{i,j}=\sum_{k=0}^m \sum_{t=0}^m d_{k,t}\circ (h^j \times \tilde{a}^{k,j})=
\sum_{k=0}^m \sum_{t=0}^mh^j\times pr_{\ast}(\tilde{a}^{k,j}\cdot d_{k,t}).\]
Therefore, since for any $0\leq k\leq m$ and for any $0\leq t\leq m$, the cycles $\tilde{a}^{k,j}$ and $d_{k,t}$ are rational, we get that $s^{i,j}$ is rational and 2) is proved.
\end{proof}

Furthermore, we fix a smooth subquadric $P$ of $Q$ of dimension $m$; we write $in$ for the imbedding 
\[(P\hookrightarrow Q)\times id_Y: P\times Y \hookrightarrow Q\times Y.\] 

\medskip

Then, considering $x$ as a correspondence, we set
\[z:=in^{\ast}(x\circ \gamma)\in CH^m(P \times Y).\]

According to the composition rules of correspondences described in \cite[\S 62]{EKM} and in view of decompositions (3.2) and (3.4),
we get that the image $\overline{z}\in CH^m(\overline{P} \times \overline{Y})$ of $z$ can be written as
\[\overline{z}=\sum_{i=0}^m \alpha_i\cdot h^{i}\times x^i\]
(we recall that the integer $\alpha_i$ is even for all odd $i$ and that $\alpha_0$ is odd). For every $i=0,...,m$, we set $z^i:=\alpha_i\cdot x^{i} \in CH^{m-i}(\overline{Y}).$

\medskip

Note that since the Steenrod operations of cohomological type commute with $in^{\ast}$ (see \cite[Theorem 61.9]{EKM}), 
for every $i=0,...,m$, the cycle $in^{\ast}(s^i)\in CH^{m+i}(\overline{P}\times \overline{Y})$ (with $s^i$ as in Lemma 3.5)
is a rational integral representative of $S^i(\overline{z}\;mod\;2)\in Ch^{m+i}(\overline{P}\times \overline{Y})$.

\begin{lemme}
\textit{For any $[(m+1)/2]\leq m'\leq m$, the cycle
\[\sum_{i=0}^{m'} {m'+i+1 \choose i} s^{m'-i,m'-i}\in CH^m(\overline{Y})\]
is the sum of a rational element $\overline{\delta_{m'}}$ and an exponent $2$ element.}
\end{lemme}

\begin{proof}
For any $[(m+1)/2]\leq m'\leq m$, we can fix a smooth subquadric $P'$ of $P$ of dimension $m'$; we write $in_{m'}$
for the imbedding 
\[(P'\hookrightarrow P)\times id_Y: P'\times Y \hookrightarrow P\times Y.\]
By \cite[Lemma 1.2]{ARC1}, one has
\[S^{m'} {pr_{m'}}_{\ast}{in_{m'}}^{\ast}(z\;mod\;2)=\sum_{i=0}^{m'} {pr_{m'}}_{\ast}(c_i(-T_{P'})\cdot {in_{m'}}^{\ast}S^{m'-i}(z\;mod\;2))\;\;\;\;\;\text{in}\;\;Ch^{m}(Y)\]
(where $T_{P'}$ is the tangent bundle of $P'$, $c_i$ are the Chern classes, and $pr_{m'}$ is the projection 
$P'\times Y \rightarrow Y$).

\medskip

If $m'\geq [(m+1)/2]+1$, since ${pr_{m'}}_{\ast}{in_{m'}}^{\ast}(z\;mod\;2) \in Ch^{m-m'}(Y)$ and $m-m'<m'$, we have $S^{m'} {pr_{m'}}_{\ast}{in_{m'}}^{\ast}(z\;mod\;2)=0$.
Therefore, we get 
\[\sum_{i=0}^{m'} {pr_{m'}}_{\ast}(c_i(-T_{P'})\cdot {in_{m'}}^{\ast}S^{m'-i}(z\;mod\;2))=0\;\;\;\;\;\text{in}\;\;Ch^{m}(Y).\]
Furthermore, by \cite[Lemma 78.1]{EKM}, for any $i=0,...,m'$, one has $c_i(-T_{P'})\equiv {m'+i+1 \choose i}h^i\;(mod\;2)$.
We deduce that the cycle 
\[\sum_{i=0}^{m'}{m'+i+1 \choose i} {pr_{m'}}_{\ast}(h^{i}\cdot {in_{m'}}^{\ast}in^{\ast}(s^{m'-i}))\]
is twice a rational element $\overline{\delta_{m'}} \in CH^m(\overline{Y})$.
Since, by the projection formula (\cite[Proposition 56.9]{EKM}), for any $i=0,...,m'$, one has ${pr_{m'}}_{\ast}(h^{i}\cdot {in_{m'}}^{\ast}in^{\ast}(s^{m'-i}))
=pr_{\ast}(h^{n-m'+i}\cdot s^{m'-i})=2s^{m'-i,m'-i}$, we are done with the case $m'\geq [(m+1)/2]+1$.

\medskip

If $m'=[(m+1)/2]$ and $m$ is odd, we still have $m-m'< m'$ and we can do the same reasoning as in the first case.
If $m'=[(m+1)/2]$ and $m$ is even, we have $m-m'=m'=m/2$, and in this case, we have
\[S^{m/2} {pr_{m/2}}_{\ast}{in_{m/2}}^{\ast}(z\;mod\;2)=({pr_{m/2}}_{\ast}{in_{m/2}}^{\ast}(z\;mod\;2))^2.\]
Therefore, by the same reasoning as in the first case, there exists $\delta_{m/2}\in CH^m(Y)$ such that
\[2\sum_{i=0}^{m/2}{\frac{m}{2}+i+1 \choose i} s^{\frac{m}{2}-i,\frac{m}{2}-i}=2\overline{\delta_{m/2}}
+({pr_{m/2}}_{\ast}{in_{m/2}}^{\ast}(\overline{z}))^2.\] 
Moreover, we have
\[({pr_{m/2}}_{\ast}{in_{m/2}}^{\ast}(\overline{z}))^2=(2z^{\frac{m}{2}})^2=2\cdot (2{z^{\frac{m}{2}}}^2),\]
and since for any $i=0,...,m$, the cycle $2z^i={pr_m}_{\ast}(h^{m-i}\cdot \overline{z})$ is rational, the cycle
\[2{z^{\frac{m}{2}}}^2={pr_m}_{\ast}({\overline{z}}^2)-4\sum \limits_{\underset{i \neq \frac{m}{2}}{0\leq i \leq m}} z^i \cdot z^{m-i}\]
is rational also and we are done with the proof of Lemma 3.7.
\end{proof}

\begin{lemme}
\textit{For any $j=0,...,m$, one can choose
an integral representative $v^j \in CH^m(\overline{Y})$ of $S^j(z^{j}\;mod\;2)$ such that}

\medskip 

\textit{1) the cycle $2v^j$ is rational;}

\medskip

\textit{2) the cycle $v^j$ is rational for odd $j$;}
 
\medskip

\textit{3) for any $k=0,...,m$, one has $s^{k,k}=\sum_{j=0}^k a_j^k v^j$, where $a_j^k$ is the binomial coefficient ${j \choose k-j}$.} 
\end{lemme}

\begin{proof}   
We induct on $j$. For $j=0$, one has $2z^0={pr_m}_{\ast}(h^m\cdot \overline{z})$, so we choose $v^0:=z^0$.
For $j=1$, one has 
\[S^1((\overline{x}\;mod\;2)\circ (\overline{\gamma}\;mod\;2))=\sum_{i=0}^m h^i\times S^1(z^i\;mod\;2)+\sum_{i=0}^m i\cdot h^{i+1}\times (z^i\;mod\;2)\in Ch^{m+1}(\overline{Q}\times \overline{Y}).\] 
In the latter expression, the coordinate on $h^1$, whose $s^{1,1}$ is an integral representative, is $S^1(z^1\;mod\;2)$.
Since, by Lemma 3.5, the cycle $s^{1,1}$ is rational, we choose $v^1:=s^{1,1}$.
Assume that the representatives $v^0, v^1,...,v^{j-1}$ are already built.

\medskip

One has
\[S^j((\overline{x}\;mod\;2)\circ (\overline{\gamma}\;mod\;2))=\sum_{k=0}^j \sum_{i=0}^m  S^k(h^i)\times S^{j-k}(z^i\;mod\;2)\in Ch^{m+j}(\overline{Q}\times \overline{Y}).\]
In the latter expression, the coordinate on $h^j$, whose $s^{j,j}$ is an integral representative, is 
\[a_j^j\cdot S^j(z^j\;mod\;2)+a_{j-1}^j\cdot S^{j-1}(z^{j-1}\;mod\;2)+\cdot \cdot \cdot+a_0^j\cdot S^0(z^0\;mod\;2),\]
where $a_i^j={i \choose j-i}$ for any $0\leq l \leq j$. Therefore, the cycle 
\[v^j:=s^{j,j}-(a_{j-1}^j\cdot v^{j-1}+\cdot \cdot \cdot+a_0^j\cdot v^0)\]
is an integral representative of $S^j(z^j\;mod\;2)$. Moreover, the element
\[2s^{j,j}=2(v^j+a_{j-1}^j\cdot v^{j-1}+\cdot \cdot \cdot+a_0^j\cdot v^0)\]
is rational by Lemma 3.5. By the induction hypothesis, we get that the cycle $2v^j$ is rational.
Furthermore, if $j$ is odd, then the cycle $s^{j,j}$ is rational by Lemma 3.5, and for any even $0\leq l \leq j$, the binomial coefficient $a_l^j$ is even. Therefore, by the induction hypothesis, we get that the cycle $v^j$ is rational. We are done with the proof of Lemma 3.8.  
\end{proof}

Finally, the following lemma will lead to the conclusion.
Denote as $\eta(X)$ the power series $\sum_{i\geq 0} \eta_i \cdot X^i$ in variable $X$, where $\eta_l=(-1)^l{2l+1 \choose l}$.

\begin{lemme}
\textit{For any polynomial $f \in \mathbb{Z}[X]$ of degree $\leq [m/2]$, the linear combination
\[\sum_{j=0}^m g_{m-j}\cdot v^j\]
is the sum of a rational element and an exponent $2$ element, where $g(X)=\sum_l g_l\cdot X^l$ is the power series $f(X)\cdot \eta(X)$.} 
\end{lemme}

\begin{proof}
Let $f=\sum f_k \cdot X^k \in \mathbb{Z}[X]$ be some polynomial of degree $\leq [m/2]$.
Consider the element 
\[\varepsilon:=\sum_{m'=[\frac{m+1}{2}]}^m f_{m-m'}\cdot \delta_{m'}\in CH^m(Y),\]
with $\delta_{m'}$ as in Lemma 3.7. Then, we have 
\[2\overline{\varepsilon}=2\sum_{m'=[\frac{m+1}{2}]}^m f_{m-m'}\sum_{i=0}^{m'}{m'+i+1 \choose i} s^{m'-i,m'-i}.\]
Furthermore, by Lemma 3.8 3), for any $k=0,...,m$, one has $s^{k,k}=\sum_{j=0}^k a_j^k v^j$.
Hence, we get the identity 
\[2\overline{\varepsilon}=2\sum_{m'=[\frac{m+1}{2}]}^m f_{m-m'} \sum_{j=0}^{m'}\left( \sum_{l=0}^{m'-j} {m'+l+1 \choose l}{j \choose m'-l-j} \right)v^j,\]
and the latter identity can be rewritten as 
\[2\overline{\varepsilon}=2\sum_{i=0}^{[\frac{m}{2}]} \sum_{j=0}^{m}f_i \cdot c_{i,j} \cdot v^j,\]
where $c_{i,j}:=\sum_{l=0}^{m-i-j}{m-i+l+1 \choose l}{j \choose m-i-j-l}$.
If $m-i-j<0$, then we have $c_{i,j}=\eta_{m-i-j}=0$. Otherwise -- if $m-i-j\geq 0$ -- we set $k:=m-i-j$, and we have
\[c_{i,j}\equiv \sum_{l=0}^k {-k-j-2 \choose l}{j \choose k-l}\;\;\;\;(mod\;2),\]
which is congruent modulo $2$ to ${-k-2 \choose k}$ by the Chu-Vandermonde Identity (see \cite[Corollary 2.2.3]{AAR}).
Therefore, since ${-k-2 \choose k}\equiv {2k+1 \choose k}\;(mod\;2)$, we get that, for any $i=0,...,[m/2]$ and for any $j=0,...,m$,
\[c_{i,j}\equiv \eta_{m-i-j}\;\;\;\;(mod\;2).\]
Thus, since by Lemma 3.8 1), for any $j=0,...m$, the cycle $2v^j$ is rational, we get that there exists an element $\delta \in CH^m(Y)$ such that
\[2\overline{\delta}=2\sum_{i=0}^{[\frac{m}{2}]} \sum_{j=0}^{m}f_i \cdot \eta_{m-i-j} \cdot v^j=2\sum_{j=0}^{m}g_{m-j}\cdot v^j,\]
where $g(X)=\sum_l g_l\cdot X^l$ is the power series $f(X)\cdot \eta(X)$.
Hence, there exists an exponent $2$ element $\lambda \in CH^m(\overline{Y})$ such that
\[\sum_{j=0}^{m}g_{m-j}\cdot v^j=\overline{\delta}+\lambda,\]
and we are done. 
\end{proof}

We finish now the proof of Theorem 3.1.
Since for any $j=0,...,m$, the cycle $2v^j$ is rational, $v^j$ is rational for all odd $j$, and $v^0$ is an odd multiple of $\overline{y}$, 
it follows from Lemma 3.9 that it is sufficient now to find  a polynomial $f \in \mathbb{Z}[X]$ of degree $\leq [m/2]$
such that the power series $g(X):=f(X)\cdot \eta(X)$ has an odd coefficient at $X^m$ and even coefficients at smaller monomials
of the same parity. By \cite[Lemma 3.13]{RIC}, such a polynomial exists. 
\end{proof}

\section{A stronger version of main theorem}

In this section we continue to use notation introduced in the beginning of Section~3.
The following result is stronger than Theorem 3.1 although its statement is less eloquent.

\medskip

Let $K/F$ be an extension and $X$ be an $F$-variety. In the following proof, an element $x\in CH^{\ast}(X_K)$ 
is called \textit{rational} if it is in the image of the restriction homomorphism $CH^{\ast}(X)\rightarrow CH^{\ast}(X_K)$. 

\medskip

In the same way as before, the following theorem is a generalization of \cite[Proposition 3.7]{RIC} 
(although, putting aside characteristic, Theorem 4.1 is still weaker than the original version 
in the sense that an exponent $2$ element appears in the conclusion). 

\begin{thm}
\textit{Assume that $m<[n/2]$ and $i_1>1$, and let $E/F$ be an extension such that $i_0(Q_E)>m$.
Then, for any $y\in CH^m(Y_{F(Q)})$ there exists $\delta\in CH^m(Y)$ and an exponent $2$ element $\lambda \in CH^m(Y_{E(Q)})$ such that
$y_{E(Q)}={\delta}_{E(Q)}+\lambda$.}
\end{thm}

\begin{proof}
We proceed the same way as in the proof of Theorem 3.1.

Let us fix an element $x \in CH^m(Q\times Y)$ mapped to $y$ under the surjection
\[CH^m(Q\times Y)\twoheadrightarrow CH^m(Y_{F(Q)}).\] 
Since $i_0(Q_E)>m$, by Remark 1.2 (applied with $r=k=m$), the image $x_{E(Q)} \in CH^m(Q_{E(Q)} \times Y_{E(Q)})$ 
of $x$ decomposes as
\[x_{E(Q)}=\sum_{i=0}^m h^i \times x^i\]
where $x^i \in CH^{m-i}(Y_{E(Q)})$ is the coordinate of $x_{E(Q)}$ on $h^i$ (see Definition 1.1).

\medskip

The image of $x$ under the composition
\[CH^m(Q\times Y)\rightarrow CH^m(Q_E\times Y_E)\rightarrow CH^m(Y_{E(Q)})\]
is $x^0$. Therefore, by the commutativity of the diagram
\[\xymatrix
{CH^m(Q_E\times Y_E)  \ar[r] & CH^m(Y_{E(Q)})\\
 CH^m(Q\times Y) \ar[u] \ar[r] & CH^m(Y_{F(Q)}) \ar[u] }
\]     
we get that $x^0=y_{E(Q)}$ and we want to prove that there exists $\delta\in CH^m(Y)$ and an exponent $2$ element $\lambda \in CH^m(Y_{E(Q)})$ such that $x^0={\delta}_{E(Q)}+\lambda$.

\medskip

Let $\pi \in Ch_{n+i_1-1}(Q^2)$ be an element mapped to the $1$-primordial cycle under the restriction homomorphism $Ch^{\ast}(Q)\rightarrow Ch^{\ast}(\overline{Q})$.
By \cite[Proposition 83.2]{EKM}, there is no cycle of type $h^j\times l_{j}$ with odd $j$ appearing in the decomposition of
 $(h^0\times h^{i_1-1})\cdot \pi_{\overline{E}(Q)}\in Ch_{n}(Q_{\overline{E}(Q)}^2)$ (and the cycle $h^0\times l_{0}$ appears).

\medskip

Moreover, since the coefficients near the cycles contained in the decomposition of 
$(h^0\times h^{i_1-1})\cdot \pi_{E(Q)}\in Ch_{n}(Q_{E(Q)}^2)$ given by Remark 1.3 (with $k=m$) 
 do not change when going over $\overline{E}(Q)$,
the cycle $(h^0\times h^{i_1-1})\cdot \pi_{E(Q)}$ can be uniquely written as 
a linear combination of cycles of type $h^j\times l_{j}$ with \textbf{even} $j\in \llbracket 0,m  \rrbracket$
(and the coefficient near $h^0\times l_{0}$ is $1$)
, of cycles of type $l_{j}\times h^j$ (where $j\in \llbracket 0,m  \rrbracket$),  and of a cycle $\rho \in Ch^{n}(Q_{E(Q)}^2)$ 
whose coordinate on $h^j$ (as well as coordinate on $l_j$)
is $0$ for $j\in \llbracket 0,m  \rrbracket$.

\medskip

Thus, fixing a rational integral representative $\gamma_{E(Q)} \in CH_{n}(Q_{E(Q)}^2)$
of $(h^0\times h^{i_1-1})\cdot \pi_{E(Q)}$, we get that the integral coefficient $\alpha_j$ near the cycle $h^j\times l_j$ contained in the 
decomposition of $\gamma_{E(Q)}$ (given by Remark 1.3 , with $k=m$), is even for all odd $j$, and that $\alpha_0$ is odd. 

\medskip

Let $\gamma \in CH_{n}(Q^2)$ mapped to $\gamma_{E(Q)}$ under the restriction homomorphism $CH_n(Q^2)\rightarrow CH_n(Q_{E(Q)}^2)$.
We have the following lemma, whose the statement and the proof are very close to Lemma 3.5.

\begin{lemme}
\textit{For any $i=0,...,m$, one can choose a rational integral representative $s^i\in CH^{m+i}(Q_{E(Q)}\times Y_{E(Q)})$ of 
$S^i((x_{E(Q)}\;mod\;2)\circ (\gamma_{E(Q)}\;mod\;2))$ such that}

\medskip

\textit{1) for any $0\leq j \leq m$, $2s^{i,j}$ is rational , where $s^{i,j}$ is the coordinate of $s^i$ on $h^j$;}

\medskip

\textit{2) for any odd $0\leq j \leq m$, $s^{i,j}$ is rational.}
\end{lemme}

\begin{proof}
We use same notation as those introduced during the proof of Lemma 3.5. One can prove 1) exactly as the same way as Lemma 3.5 1).
We need the following proposition to prove 2).

\begin{prop}
\textit{Let $X$ be a smooth $F$-variety and let $\rho$ be an element of
$Ch(Q\times X)$ such that for any $j=0,...,r$, its coordinate $\rho^j$ on $h^j$ is $0$.
Then, for any integer $k$ and for any $j=0,...,r$, the coordinate of $S^k(\rho)$ on $h^j$ is $0$.}
\end{prop}

\begin{proof}
We induct on $k$. For $k=0$, one has $S^0=\text{Id}$.
Assume that the statement is true till the rank $k$ and let $j\in \llbracket 0,r  \rrbracket$.
By \cite[Corollary 61.15]{EKM} (Cartan Formula), one has
\[S^{k+1}(l_j\cdot \rho)=l_j\cdot S^{k+1}(\rho)+\sum_{i=1}^{k+1}S^i(l_j)\cdot S^{k+1-i}(\rho).\]
Since for any $i=1,...,k+1$, the cycle $S^i(l_j)$ is a multiple of $l_{j-i}$ (see \cite[Corollary 78.5]{EKM}), 
by the induction hypothesis, we get
\[pr_{\ast}(l_j\cdot S^{k+1}(\rho))=pr_{\ast}(S^{k+1}(l_j\cdot \rho)).\]
Furthermore, by \cite[Proposition 61.10]{EKM}, one has
\[S^{k+1}\circ pr_{\ast}(l_j\cdot \rho)=\sum_{i=0}^{k+1}pr_{\ast}(c_{k+1-i}(-T_Q)\cdot S^i(l_j\cdot \rho)),\]
and since $pr_{\ast}(l_j\cdot \rho)=0$, we deduce that
\[pr_{\ast}(l_j\cdot S^{k+1}(\rho))=\sum_{i=0}^k a_i,\]
where $a_i=pr_{\ast}(c_{k+1-i}(-T_Q)\cdot S^i(l_j\cdot \rho))$. We are going to prove that for any $i=0,...,k$, one has $a_i=0$.
Let $i$ be an integer in $\llbracket 0,k  \rrbracket$. Since by \cite[Lemma 78.1]{EKM}, the cycle $c_{k+1-i}(-T_Q)$ is a multiple 
of $h^{k+1-i}$, it suffices to show that $pr_{\ast}(h^{k+1-i}\cdot S^i(l_j\cdot \rho))=0$.

By the Cartan Formula and \cite[Corollary 78.5]{EKM}, the cycle $pr_{\ast}(h^{k+1-i}\cdot S^i(l_j\cdot \rho))$ is a linear combination
of cycle of type $pr_{\ast}(h^{k+1-i}\cdot l_{j-t}\cdot S^{i-t}(x))$, where $t\in \llbracket 0,i  \rrbracket$.
Since by \cite[Proposition 68.1]{EKM}, for any $t=0,...,i$, one has $h^{k+1-i}\cdot l_{j-t}=l_{j-t-(k+1-i)}$, we are done by the induction hypothesis.
\end{proof}

We finish now the proof of Lemma 4.2. Assume that $j$ is odd.
Since by Proposition 4.3, for any $k=0,...,m$, the coordinate of $S^k(\rho)$ on $h^j$ is $0$,
the only fact that we have to explain here to prove 2) 
(i.e what is new compared to the proof of Lemma 3.5) is why the corresponding cycle $\tilde{a}^{k,j}\in CH^{n+k-j}(Q_{E(Q)})$ is rational.

\medskip

For the same reasons as in the proof of Lemma 3.5, the cycle $\tilde{a}^{k,j}\in CH^{n+k-j}(Q_{E(Q)})$ is divisible by $2$.
Moreover, since one has $j-k\leq m<i_0(Q_E)$, 
the cycle $l_{j-k}$ is defined over $E$ and it is consequently defined over $E(Q)$.
Furthermore, since $j-k\leq m<[n/2]$, the group $CH^{n+k-j}(Q_{E(Q)})$ is free with basis $\{l_{j-k}\}$ (as well as the group $CH^{n+k-j}(Q_{\overline{E(Q)}})$) and therefore the restriction homomorphism 
\[CH^{n+k-j}(Q_{E(Q)})\longrightarrow CH^{n+k-j}(Q_{\overline{E(Q)}})\] 
is injective (it is even an isomorphism).
Since $2l_{j-k}=h^{n+k-j}$, we deduce that
any cycle of $CH^{n+k-j}(Q_{E(Q)})$ divisible by $2$ is rational.
Thus, for any $0\leq k \leq m$, the cycle $\tilde{a}^{k,j}$ is rational and we finish as in the proof of Lemma 3.5. 
\end{proof}

Now, one can finish the proof of Theorem 4.1 exactly the same way as the proof of Theorem 3.1 replacing $\overline{F}$ by $E(Q)$.
\end{proof}

\bibliographystyle{acm} 
\bibliography{biblio}
\end{document}